\documentclass[12pt,reqno]{amsart}

\theoremstyle{plain}
\newtheorem{theorem} {Theorem}[section]
\newtheorem{lemma}[theorem] {Lemma}

\newtheorem{corollary}[theorem] {Corollary}

\theoremstyle{definition}
\newtheorem{definition}[theorem] {Definition}

\theoremstyle{remark}
\newtheorem{remark}[theorem] {Remark}

\numberwithin{equation}{section}

\usepackage{verbatim} %paulN
\usepackage{color}    %paulN

\def\DEF{\stackrel{\rm def}{=}}

\def\CC{\mathbb{C}}

\def\NN{\mathbb{N}}
\def\PP{\mathbb{P}}
\def\RR{\mathbb{R}}
\def\TT{\mathbb{T}}

\def\oCC{\operatorfont{C\/}}
\def\oBB{\operatorfont{B\/}}
\def\omm{\operatorfont{m\/}}

\newcommand{\mcAA}{{\mathcal A}}

\def\({\left(}                  \def\){\DOTSX\right)}
\def\lbr{\left[}                \def\rbr{\DOTSX\right]}
\def\lv{\left|}                 \def\rv{\DOTSX\right|}
\def\lb{\left\{}                \def\rb{\DOTSX\right\}}
\def\lno{\left\|}               \def\rno{\DOTSX\right\|}
\def\<{\left\langle}            \def\>{\right\rangle}

\begin{document}

\parindent=0pt\parskip=8pt
\linespread{1.3}

\title[M.~Riesz-Schur-type inequalities]
{M.~Riesz-Schur-type inequalities \\ for entire functions of exponential type}

\author[T.~Erd\'elyi]{Tam\'as Erd\'elyi}
\address{Department of Mathematics, Texas A\&M University, College Station, TX 77843, USA}
\email{terdelyi@math.tamu.edu }

\author[M.~I.~Ganzburg]{Michael I. Ganzburg}
\address{Department of Mathematics, Hampton University, Hampton, VA 23668, USA}
\email{michael.ganzburg@hamptonu.edu}
 
\author[Paul Nevai]{Paul Nevai}
\address{KAU, Jeddah, Saudi Arabia, \texttt{and} Upper Arlington (Columbus), Ohio, USA}
\email{paul@nevai.us}
\thanks{The research of Paul Nevai was supported by KAU grant No. 20-130/1433 HiCi.}



\begin{abstract} 

We prove a general M.~Riesz-Schur-type inequality for entire functions of
exponential type. If $f$ and $Q$ are two functions of exponential type
$\sigma > 0$ and $\tau \ge 0$, respectively, and if $Q$ is real-valued and the
real and distinct zeros of $Q$ are bounded away from each other, then
\begin{equation*}
	|f(x)| \le (\sigma + \tau) (A_{\sigma + \tau}(Q))^{-1/2} \lno Q f \rno_{\oCC(\RR)},
	\qquad x\in\RR \, ,
\end{equation*}
where $A_s(Q) \DEF \inf_{x\in\RR} \(\lbr Q'(x) \rbr^2 + s^2 \lbr Q(x) \rbr^2\)$. 
We apply this inequality to the weights 
$Q(x) \DEF \sin \(\tau x\)$ and $Q(x) \DEF x$ and describe all extremal
functions in the corresponding inequalities. 

\end{abstract}

%\dedicatory{Dedicated to Andrei A. Gonchar and Herbert Stahl}

\date{submission: April 15, 2014}

%; last revision: MMM DD, 2014; accepted: MM DD, 2014; available online: MM DD, 2014}

\maketitle

\section{Introduction}\label{S1}

\setcounter{equation}{0}

For $n\in\NN$, let $\PP_n$ and $\TT_n$ denote the set of algebraic and
trigonometric polynomials of degree at most $n$ with complex coefficients,
respectively. Next, given $\gamma \ge 0$, let $\oBB_\gamma$ be the collection
of entire functions of order $1$ and type at most $\gamma$, that is, entire
functions of exponential type at most $\gamma$. In other words,
$f\in\oBB_\gamma$ if and only if for all $\varepsilon > 0$ and for all
$z\in\CC$ we have $\limsup_{|z|\to\infty} |f(z)| \exp\(-\(\gamma +
\varepsilon\)|z|\) < \infty$. Finally, let $\oCC(\Omega)$ be the space of all
continuous complex-valued functions $f$ on $\Omega \subset \RR$ for which
$\lno f \rno_{\oCC(\Omega)} \DEF \sup_{t\in\Omega} |f(t)| <
\infty$.\footnote{With apologies for the repulsive $\lno f(t)
\rno_{\oCC(\Omega)}$ notation that we occasionally use in this paper.}

The M.~Riesz-Schur Inequality\footnote{Hitherto, this used to be called Schur
Inequality; see, e.g., \cite{nevai2014} and \cite{nevai201x} for reasons why
one should add Marcel Riesz's name to it.}
\begin{equation} \label{eq_schur}
	\lno P \rno_{\oCC([-1,1])} \le 
	(n+1) \, \lno \sqrt{1-t^2} \, P(t) \rno_{\oCC([-1,1])} ,
	\qquad P\in\PP_{n} \, ,
\end{equation}
is well known and it plays an essential role in approximation theory. In
particular, the simplest proof of Markov's inequality is based on applying
\eqref{eq_schur} to the Bernstein Inequality. We have no doubts that every
person reading this paper has seen this done on countless occasions so let us
just name one such reference here, namely, M.~Riesz's \cite[\S3, Satz~IV,
p.~359]{mriesz1914_2} from 1914 which is the granddaddy of all of them.

The following M.~Riesz-Schur-type
\begin{equation} \label{E1.1}
	\lno T \rno_{\oCC([-\pi,\pi))} \le
	(n+1) \lno \sin t \, T(t) \rno_{\oCC([-\pi,\pi))},
	\qquad T\in\TT_n \, ,
\end{equation}
and Schur-type
\begin{equation} \label{E1.2}
	\lno P \rno_{\oCC([-1,1])} \le
	(n+1) \lno t \, P(t) \rno _{\oCC([-1,1])}, 
	\qquad P\in\PP_n \, ,
\end{equation}
inequalities are less known. In fact, polling top experts, the third of us
found not a single person who was aware of \eqref{E1.1} despite looking so
deceivingly similar to \eqref{eq_schur}. Although for even trigonometric
polynomials \eqref{E1.1} is indeed equivalent to \eqref{eq_schur}, we don't
know if any simple trick could be used to extend it from even to all
trigonometric polynomials. We found \eqref{E1.1} hiding as
\cite[Lemma~15.1.3, p.~567]{rahman2002} with no reference to any original
source, and, therefore, we must assume that \cite{{rahman2002}} is the one
and only place where the proof is presented and there is a good probability
that the proof belongs to Q.~I.~Rahman.\footnote{Qazi Ibadur Rahman died on
July 21, 2013, and, thus, the source of \eqref{E1.1} may never be found.}
Inequality \eqref{E1.2} was proved in a 1919 paper by Schur, see \cite[(29),
p.~285]{schur1919}. As a matter of fact, Schur proved a little more. Namely,
for odd $n$ the factor $n+1$ in \eqref{E1.2} can be replaced by $n$.

In this note we discuss various versions of inequalities \eqref{E1.1} and
\eqref{E1.2} for entire functions of exponential type. They are special cases
of the following general M.~Riesz-Schur-type inequality.

\begin{theorem} \label{T1.1}

Let $\tau \ge 0$ and $\sigma > 0$. Let  $Q\in \oBB_\tau$ be a function that is
real-valued on $\RR$ and let $Z_Q \DEF \{t_n\}$ denote the real and distinct
zeros of $Q$. Assume that if $\operatorname{card} Z_Q > 1$ then $\inf_{k\ne
\ell} \lv t_k-t_\ell \rv > 0$. Let
\begin{equation} \label{E1.31}
	A_s(Q) \DEF \inf_{t\in\RR} \(\lbr Q'(t) \rbr^2 + s^2 \lbr Q(t) \rbr^2\) ,
	\qquad s \ge 0 \, .
\end{equation}
Then the following two statements hold true.

{\rm(a)} If $f\in \oBB_\sigma$, then we have
\begin{equation} \label{E1.5}
	|f(x)| \le (\sigma + \tau) \(A_{\sigma + \tau}(Q)\)^{-1/2} \lno Q f \rno_{\oCC(\RR)},
	\qquad x \in \RR \, .
\end{equation}

{\rm(b)} If equality holds in \eqref{E1.5} for a point $x = x_0 \in \RR$ and
for a function $f \in \oBB_\sigma$ that is real-valued on $\RR$ with $f
\not\equiv 0$, then either $ \lv (Qf)(x_0) \rv = \lno Qf \rno_{\oCC(\RR)}$ or
there exist two real constants $S$ and $C$ such that $|S| + |C| > 0$ and
\begin{equation} \label{E1.51}
	Q(x) \, f(x) \equiv
	S \sin \((\sigma + \tau) \, x\) + C \cos \((\sigma + \tau) \, x\) ,
	\qquad x \in \RR \, .
\end{equation}

\end{theorem}

\begin{remark} \label{R1.2}

Our proof of Theorem~\ref{T1.1} is based on the Duffin-Schaeffer inequality
for entire functions of exponential type, see \cite[Theorem~II,
p.~555]{DS1937}. The proof of Theorem~\ref{T1.1}(a) is similar to the
proof of \cite[Lemma~15.1.3, p.~567]{rahman2002} in Rahman-Schmei{\ss}er's
book.

\end{remark}

\begin{remark} \label{R1.3}

If equality holds in \eqref{E1.5} for a point $x = x_0 \in \RR$ and for a
function $f \in \oBB_\sigma$ that is real-valued on $\RR$, then $f'(x_0) = 0$
so that if $f \not\equiv 0$ and $Q'(x_0) \ne 0$, then $\lv (Qf)(x_0) \rv <
\lno Qf \rno$.

\end{remark}

%\medskip\centerline{{\color{blue}\fbox{\fbox{\bf $\Longrightarrow$ Michael \& Tam\'as: START READING HERE $\Longleftarrow$}}}}%hereiam

\begin{remark} \label{R1.4}

Theorem~\ref{T1.1}(b) shows that equality holds in \eqref{E1.5} only in
rather exceptional cases since $f$ in \eqref{E1.51} must be
an entire function (of order $1$ and type at most $\sigma$) so that the zeros
of $Q$, counting multiplicities, must be zeros of the right-hand side as well.

\end{remark}

In the following two corollaries we apply Theorem~\ref{T1.1} with the two
M.~Riesz-Schur's weights that appear in \eqref{E1.1} and \eqref{E1.2}.

\begin{corollary} \label{C1.2}

Let $\tau > 0$ and $\sigma > 0$. If $f\in \oBB_\sigma$, then
\begin{equation} \label{E1.6}
	|f(x)| \le \( \frac{\sigma} {\tau} + 1 \) \lno \sin \(\tau \, t\) f(t) \rno_{\oCC(\RR)},
	\qquad x\in\RR \, .
\end{equation}
Moreover, equality holds in \eqref{E1.6} for a point $x = x_0 \in \RR$ such
that $\cos \(\tau \, x_0\) \ne 0$, and for a function $f\in
\oBB_\sigma$ that is real-valued on $\RR$ with $f \not\equiv 0$ if and only
if $\sigma/\tau \in \NN$ and 
$f(x) \equiv S \sin \(\(\sigma + \tau\) x\)/\sin \(\tau x\)$ for some real 
$S \ne 0$.

\end{corollary}

\begin{corollary} \label{C1.3}

Let $\sigma > 0$. If $f\in \oBB_\sigma$, then
\begin{equation} \label{E1.7}
	|f(x)| \le \sigma \lno t f(t) \rno_{\oCC(\RR)},
	\qquad x\in\RR \, .
\end{equation}
Moreover, equality holds  in \eqref{E1.7} for a point $x=x_0 \in \RR$ and for
a function $f\in \oBB_\sigma$ that is real-valued on $\RR$  with $f
\not\equiv 0$ if and only if $f(x) \equiv S \sin \(\sigma \, x\)/x$ for some real
$S \ne 0$.

\end{corollary}

First, we discuss some lemmas in Section~\ref{S2}, and then we give the proofs
of these results in Section~\ref{S3}.

Throughout Sections~\ref{S2} \& \ref{S3}, we will assume that $\tau \ge 0$,
$\sigma > 0$, and that the function $Q\in \oBB_\tau$ is real-valued on $\RR$.
When necessary, will refer to the conditions
\begin{equation} \label{E1.3}
	A_s(Q) > 0
	\quad \mbox{where} \quad
	A_s(Q) \DEF \inf_{t\in\RR} \(\lbr Q'(t) \rbr^2 + s^2 \lbr Q(t) \rbr^2\) ,
	\qquad s \ge \theta \, ,
\end{equation}
where $\theta \ge \tau$ is fixed, and
\begin{equation} \label{E1.4}
	d > 0
	\quad \mbox{where} \quad
	d \DEF \lb
	\begin{array}{ll}
		\inf_{k\ne \ell}|t_k-t_\ell|, 
			&\operatorname{card} Z_Q \ge 2 \, , \\
		1,
			&\mathrm{otherwise} \, ,
	\end{array}
	\right.
\end{equation}
where $Z_Q \DEF \{t_n\}$ denote the real and distinct zeros of $Q$.

In what follows, the Lebesgue measure of a (Lebesgue) measurable set
$E\subset\RR$ will be denoted by $\omm(E)$.

\section{Some Lemmas}\label{S2}

\setcounter{equation}{0}

We will discuss properties of $Q$ and of other entire functions of
exponential type that are needed for the proof of Theorem~\ref{T1.1}. We
start with two simple properties.

\begin{lemma} \label{L2.1}

$\phantom{.}$

%If $Q$ satisfies condition \eqref{E1.3}, then the following statements hold:

{\rm(a)} $A_s(Q)$ defined by \eqref{E1.31} is a right-continuous nondecreasing function of $s\in [0,\infty)$.

{\rm(b)} If $Q$ satisfies condition \eqref{E1.3}, then $\limsup_{x\to\infty} |Q(x)| > 0$
and $\limsup_{x\to-\infty} |Q(x)| > 0$.

\end{lemma}

\begin{proof}

$\phantom{.}$

(a) Clearly, $A_s(Q)$ is a nondecreasing function of $s \in [0,\infty)$. Let
$c \in [0,\infty)$ be a fixed number. Then, for each $\varepsilon > 0$, there
exists $u_0 = u_0(\varepsilon,c) \in \RR$ such that
\begin{equation*}
	A_{c}(Q) \ge 
	\lbr Q'(u_0) \rbr^2 + c^2 \lbr Q(u_0) \rbr^2 - \varepsilon \, ,
\end{equation*}
and, therefore,
\begin{equation*} %\label{E2.1}
	A_{c}(Q) \le 
	A_s(Q) \le 
	\lbr Q'(u_0) \rbr^2 + s^2 \lbr Q(u_0) \rbr^2 \le 
	A_{c}(Q) + (s^2-c^2) \lbr Q(u_0) \rbr^2 + \varepsilon \, ,
	\qquad s\in (c, \infty) \, ,
\end{equation*}
from which
\begin{equation*}
	A_{c}(Q) \le 
	\lim_{s \to c+} A_s(Q) \le 
	A_{c}(Q) + \varepsilon \, ,
\end{equation*}
and, letting here $\varepsilon \to 0+$, proves right-continuity of $A_s(Q)$ at $c$.

(b) Fix $s \ge \theta$. If $\limsup_{x\to\infty} |Q(x)| = 0$ , i.e.,
$\lim_{x\to\infty} Q(x) = 0$, then, for $x > 0$ and by the mean value
theorem, 
\begin{equation*}
	|x| \inf_{t\in[x,\infty)} |Q'(t)| \le |Q(2x) - Q(x)| \, ,
\end{equation*}
so that, dividing both sides by $|x|$ and letting $x\to\infty$, we
get\footnote{Actually, $Q(x) = o(|x|)$ as $x\to\infty$ is sufficient here.}
\begin{equation*}
	\liminf_{x\to\infty} |Q'(x)| \le 
	\lim_{x\to\infty} \frac {|Q(2x) - Q(x)|}{|x|} = 
	0 \, ,
\end{equation*}
and then 
\begin{equation*}
	\liminf_{x\to\infty} \(\lbr Q'(x) \rbr^2 + s^2 \lbr Q(x) \rbr^2\) \le 
	\liminf_{x\to\infty} \lbr Q'(x) \rbr^2 + s^2 \lim_{x\to\infty} \lbr Q(x) \rbr^2  = 
	0 \, ,
\end{equation*}
that contradicts \eqref{E1.3}. If $\limsup_{x\to-\infty} |Q(x)| = 0$, then
apply the previous part to $Q(-x)$ that also satisfies the conditions.
\end{proof}

Next, we discuss some metric properties of $Q$ and of other entire functions of
exponential type.

\begin{definition} \notag %\label{D2.2}

Let $0 < \delta \le L < \infty$. A measurable set $E\subseteq \RR$ is called
$(L,\delta)$-dense if for every interval $\Delta \subset \RR$ with $m(\Delta)
= L$, we have $m\(E \cap \Delta\) \ge \delta$.

\end{definition}

\begin{lemma} \label{L2.3}

Let $0 < \delta \le L < \infty$ and $\sigma>0$. If the measurable set
$E\subseteq \RR$ is $(L,\delta)$-dense, then there is a constant $C_1 > 0$
such that for every $f\in \oBB_\sigma$ the Remez-type inequality
\begin{equation} \label{E2.3}
	\lno f \rno_{\oCC(\RR)} \le C_1 \lno f \rno_{\oCC(E)}
\end{equation}
holds.

\end{lemma}

Lemma~\ref{L2.3} is a special case Katsnelson's \cite[Theorem~3,
p.~43]{katsnelson1973} where he proved it for polysubharmonic functions on
$\CC^n$.\footnote{Victor Katsnelson informed us that B.~Ja.~Levin knew of
this result, at least for $n=1$, and lectured about it way before 1971
although he published it in the form of a Remez-type inequality only much
later.}

\begin{lemma} \label{L2.4}

Let $\tau \ge 0$ and $\theta \ge \tau$. If  $Q \in \oBB_\tau$ satisfies
conditions \eqref{E1.3} and \eqref{E1.4}, then there exists $\alpha^\ast > 0$
such that the set $E^\ast \DEF \lb x\in\RR: |Q(x)| \ge \alpha^\ast \rb$ is
$(d/2,d/4)$-dense, where $d$ is defined in \eqref{E1.4}.

\end{lemma}

\begin{proof}

Let $s \ge \theta$, see \eqref{E1.3}. It follows immediately from
\eqref{E1.3} that there exists a constant $C_2 > 0$ such that 
\begin{equation} \label{E2.4} 
	\lbr Q'(x) \rbr^2 + s^2 \lbr Q(x) \rbr^2 \ge  C_2^2,
	\qquad x\in\RR \, . 
\end{equation} 
Next, given $\alpha > 0$, the open set $E_\alpha \DEF \{x\in\RR : |Q(x)| <
\alpha\}$ can be represented as $E_\alpha=\bigcup_{I\in\mcAA_\alpha} I$,
where $\mcAA_\alpha$ is a family of pairwise disjoint open intervals $I$.
Then \eqref{E2.4} shows that for $\alpha\in(0,C_2/s]$ the set $E_\alpha$ does
not contain  zeros of $Q'$. Therefore, $Q$ is strictly monotone on each
interval $I\in\mcAA_\alpha$.

If  $\alpha\in(0,C_2/s]$ and $I\in\mcAA_\alpha$ is a bounded interval then $|Q|$
takes the same value at the endpoints of $I$ so that $I$ contains precisely
one zero of $Q$.

If  $\alpha\in(0,C_2/s]$ and $I\in\mcAA_\alpha$ is an unbounded interval, say
$(a,\infty)$, then, $Q$ being monotone on $I$, the finite or infinite $\Gamma
\DEF \lim_{x\to\infty} Q(x)$ exists and, by Lemma~\ref{L2.1}(b), $\Gamma \ne 
0$. Hence, if $\alpha > 0$ is sufficiently small, say, $0 < \alpha < \alpha_1
\le C_2/s$ then $I\not\subset E_\alpha$ so that $I\not\in\mcAA_\alpha$. In
other words, if $0 < \alpha < \alpha_1 \le C_2/s$ then $\mcAA_\alpha$ has no
unbounded components.

Summarizing the above, there exists $\alpha_1\in(0,C_2/s]$ such that, for all
$\alpha\in(0,\alpha_1]$, each interval $I\in\mcAA_\alpha$ is bounded and it
contains precisely one zero $t_I$ of $Q$. In particular, if $Q$ has no real
zeros, then $E_\alpha=\emptyset$ for $\alpha\in(0,\alpha_1]$.

It remains to estimate the length of each interval $I \in \mcAA_\alpha$. Let
$\alpha_2\in(0,C_2/s]$ be the unique positive solution of the equation
$\alpha/\sqrt{C_2^2-\alpha^2 s^2} = d/8$. 
Let $\alpha^\ast \DEF \min \lb \alpha_1,\alpha_2 \rb$.
Taking account of \eqref{E2.4} and using the mean value theorem, we see that
for each $x \in I$ there exists $\xi \in I$ such that
\begin{equation*}
	|x-t_I| = 
	\frac {|Q(x)|} {|Q'(\xi)|} \le 
	\frac {\alpha} {\sqrt{C_2^2-\alpha^2 s^2}} \le d/8 ,
\end{equation*}
whenever $0 < \alpha \le \alpha^\ast$. Hence $m(I) \le d/4$ for
$\alpha\in(0,\alpha^\ast]$. Therefore, due to condition \eqref{E1.4}, for every
closed interval $\Delta$ of length $d/2$, we have $m\(\Delta\cap
E_{\alpha^\ast}\) \le d/4$. This shows that $E^\ast \DEF \RR\setminus
E_{\alpha^\ast}$ is $(d/2,d/4)$-dense.
\end{proof}

Now we are in the position to prove the following Schur-type inequality.

\begin{lemma} \label{L2.5}

Let $\tau \ge 0$ and $\theta \ge \tau$. Let  $Q \in \oBB_\tau$ satisfy
conditions \eqref{E1.3} and \eqref{E1.4}. Let $\sigma > 0$. Then there is
a constant $C_3 > 0$ such that for every $f\in \oBB_\sigma$ we have 
\begin{equation}  \label{E2.5}
	\lno f \rno_{\oCC(\RR)} \le  C_3 \lno Q f \rno_{\oCC(\RR)} .
\end{equation}

\end{lemma}

\begin{proof}

Let the number $\alpha^\ast > 0$ and the $(d/2,d/4)$-dense set $E^\ast
\subset \RR$ be chosen with the help of Lemma~\ref{L2.4}. Then
\begin{equation} %\label{E2.6}
	\lno f \rno_{\oCC(E^\ast)} \le \frac {1}{\alpha^\ast} \lno Qf \rno_{\oCC(\RR)}
\end{equation}
so that \eqref{E2.5} follows from Lemma~\ref{L2.3}.
\end{proof}

\section{Proofs of the Main Results}\label{S3}

\setcounter{equation}{0}

\begin{proof}[Proof of Theorem~\ref{T1.1}(a)]

Let $\tau$, $\sigma$, and $Q$ satisfy the conditions of Theorem~\ref{T1.1}.
Let $f\in \oBB_\sigma$ and assume $f \not\equiv 0$ and $\lno Q f
\rno_{\oCC(\RR)} < \infty$ since otherwise all claims are self-evident. 

The proof will consist of three short steps. Using Lemma~\ref{L2.5} with
$\theta = \sigma + \tau$, we obtain $\lno f \rno_{\oCC(\RR)} < \infty$, and
this will be used in all three steps.

\textit{Step~1a.} First, we assume that $f$ is real-valued on $\RR$ and that
there exists $x^\ast\in\RR$ such that $|f(x^\ast)|= \lno f \rno_{\oCC(\RR)}$.
Since $Qf\in \oBB_{\sigma + \tau}$, we can apply the Duffin-Schaeffer
inequality to $Q f$ to get
\begin{equation} \label{E3.1}
	\lbr (Qf)'(x) \rbr^2 + \(\sigma + \tau\)^2 \lbr (Q f)(x) \rbr^2
	\le 
	(\sigma + \tau)^2 \lno Q f \rno_{\oCC(\RR)}^2 ,
	\qquad x\in\RR \, ,
\end{equation}
see \cite[Theorem~II, p.~555]{DS1937}. Since $f'(x^\ast) = 0$,
we can set $x = x^\ast$ in \eqref{E3.1} to arrive at 
\begin{equation*} %\label{E3.2}
	\lno f \rno_{\oCC(\RR)}^2=f^2(x^\ast) \le 
	\frac 
		{(\sigma + \tau)^2 \lno Q f \rno_{\oCC(\RR)}^2}
		{\lbr Q'(x^\ast) \rbr^2 + (\sigma + \tau)^2 \lbr Q(x^\ast) \rbr^2} \, ,
\end{equation*}
that implies \eqref{E1.5} immediately.

\textit{Step~2a.} Second, we assume that $f$ is real-valued on $\RR$ but
$|f|$ doesn't necessarily take its maximum value on $\RR$. Given $\varepsilon
> 0$, pick $x_\varepsilon \in \RR$ such that $\lno f \rno_{\oCC(\RR)} \le
\lv f(x_\varepsilon) \rv + \varepsilon$. Next, observe that, for fixed $\delta > 0$,
the function $F_{\delta}$ defined by
\begin{equation*}
	F_{\delta}(x) \DEF
	f(x) \, \frac{\sin \( \delta(x-x_\varepsilon)\)} {\delta(x-x_\varepsilon)} \, ,
	\qquad x\in\RR \, ,
\end{equation*}
belongs to $\oBB_{\sigma+\delta}$, it is real-valued on $\RR$, $\lv
F_{\delta}(x) \rv \le |f(x)|$ for $x\in\RR$, and, since $F_{\delta}(x)$ goes to
$0$ as $x\to\pm\infty$, it takes its maximal value at some point in $\RR$ so
that, as proved in Step~1a, we can use \eqref{E1.5} to obtain
\begin{align*}
	\lno f \rno_{\oCC(\RR)} - \varepsilon \le
	\lv f(x_\varepsilon) \rv &=
	\lv F_{\delta}(x_\varepsilon) \rv \le 
	\lno F_{\delta} \rno_{\oCC(\RR)} \le
	\\
	(\sigma + \tau + \delta) \(A_{\sigma + \tau+\delta}(Q)\)^{-1/2} \lno Q F_{\delta} \rno_{\oCC(\RR)} &\le
	(\sigma + \tau + \delta) \(A_{\sigma + \tau+\delta}(Q)\)^{-1/2} \lno Q f \rno_{\oCC(\RR)} .
\end{align*}
First, letting here $\varepsilon \to 0+$, we get
\begin{equation*} %\label{E3.3}
	\lno f \rno_{\oCC(\RR)} \le 
	(\sigma + \tau + \delta) \(A_{\sigma + \tau+  \delta}(Q)\)^{-1/2} \lno Q f \rno_{\oCC(\RR)} ,
\end{equation*}
and then, letting $\delta \to 0+$ and using right-continuity of $A_s(Q)$ as
proved in Lemma~\ref{L2.1}(a), we obtain \eqref{E1.5} for $f$ as well.

\textit{Step~3a.} Third, we assume that $f$ is not necessarily real-valued on
$\RR$. Then $f$ can be written as $f = f_1 +i f_2$, where both $f_1$ and
$f_2$ given by
\begin{equation*}
	f_1(z) \DEF \frac {1}{2} \(f(z) + \overline{f(\overline{z})}\)
	\quad \& \quad
	f_2(z) \DEF  \frac {1}{2i} \(f(z) - \overline{f(\overline{z})}\)
\end{equation*}
are real-valued functions on $\RR$ that belong to $\oBB_\sigma$.

Given $\varepsilon > 0$, pick $x_\varepsilon \in \RR$ such that
$f(x_\varepsilon) \ne 0$ and $\lno f \rno_{\oCC(\RR)} \le \lv
f(x_\varepsilon) \rv + \varepsilon$. Let $\eta \in \RR$ be defined by
$\exp(i\eta) = f(x_\varepsilon)/|f(x_\varepsilon)|$, and let the function $G$
be defined by $G \DEF \cos \eta \, f_1 + \sin \eta \, f_2$. Then $G$ is
real-valued on $\RR$ and satisfies the relations $|G(x)| \le |f(x)|$ for
$x\in\RR$ and $\lv G(x_\varepsilon) \rv = \lv f(x_\varepsilon) \rv$. Applying
\eqref{E1.5} to $G$, we obtain
\begin{align*}
	\lno f \rno_{\oCC(\RR)} - \varepsilon \le
	\lv f(x_\varepsilon) \rv &=
	\lv G(x_\varepsilon) \rv \le 
	\lno G \rno_{\oCC(\RR)} \le
	\\
	(\sigma + \tau) \(A_{\sigma + \tau}(Q)\)^{-1/2} \lno Q G \rno_{\oCC(\RR)} &\le
	(\sigma + \tau) \(A_{\sigma + \tau}(Q)\)^{-1/2} \lno Q f \rno_{\oCC(\RR)} ,
\end{align*}
and, letting $\varepsilon\to 0+$, we again arrive at \eqref{E1.5} for $f$ as well. 

Thus, Theorem~\ref{T1.1}(a) has fully been proved.
\end{proof}

\begin{proof}[Proof of Theorem~\ref{T1.1}(b)]

Let $\tau$, $\sigma$, and $Q$ satisfy the conditions of Theorem~\ref{T1.1}.
If there exists $x = x_0 \in \RR$ and $f \in \oBB_\sigma$ that is real-valued
on $\RR$ such that equality holds in \eqref{E1.5}, then $\lno Q f
\rno_{\oCC(\RR)} < \infty$ and $|f(x_0)|= \lno f \rno_{\oCC(\RR)} < \infty$
so that $f'(x_0) = 0$. Hence, using again \eqref{E1.5} and the definition of $A_s(Q)$ in \eqref{E1.31},
\begin{align*} %\label{E3.4}
	& \lbr (Qf)'(x_0) \rbr^2 + \(\sigma + \tau\)^2 \lbr (Q f)(x_0) \rbr^2 =
	\( \lbr Q'(x_0) \rbr^2 + \(\sigma + \tau\)^2 \lbr Q(x_0) \rbr^2 \) 
		\lno f \rno_{\oCC(\RR)}^2 =
	\\
	& (\sigma + \tau)^2 
		\frac 
		{\lbr Q'(x_0) \rbr^2 + \(\sigma + \tau\)^2 \lbr Q(x_0) \rbr^2}
		{A_{\sigma + \tau}(Q)}
		\lno Q f \rno_{\oCC(\RR)}^2 \ge
	(\sigma + \tau)^2||Qf||_{\oCC(\RR)}^2 .
\end{align*}
Summarizing,
\begin{equation*}
	\lbr (Q f)'(x_0) \rbr^2 + \(\sigma + \tau\)^2 \lbr (Q f)(x_0) \rbr^2 \ge
	(\sigma + \tau)^2||Qf||_{\oCC(\RR)}^2 ,
\end{equation*}
and, by the Duffin-Schaeffer inequality,
\begin{equation*}
	\lbr (Q f)'(x_0) \rbr^2 + \(\sigma + \tau\)^2 \lbr (Q f)(x_0) \rbr^2 \le
	(\sigma + \tau)^2||Qf||_{\oCC(\RR)}^2 ,
\end{equation*}
see \eqref{E3.1}, so that
\begin{equation*}
	\lbr (Q f)'(x_0) \rbr^2 + \(\sigma + \tau\)^2 \lbr (Q f)(x_0) \rbr^2 =
	(\sigma + \tau)^2||Qf||_{\oCC(\RR)}^2 .
\end{equation*}
If $ \lv (Q f)(x_0) \rv < \lno (Qf) \rno$, then, this equality is possible if and only
if $Qf$ is of the form 
$(Q f)(x) = S \sin \((\sigma + \tau) \, x\) + C \cos \((\sigma + \tau) \, x\)$
with some real constants $S$ and $C$, see \cite[Theorem~II, p.~555]{DS1937},
and, since $f \not\equiv 0$, we have $|S| + |C| > 0$.

Thus, Theorem~\ref{T1.1}(b) has fully been proved as well. 
\end{proof}

\begin{proof}[Proof of Corollaries~\ref{C1.2} and~\ref{C1.3}]

The weights $Q(x) \DEF \sin \(\tau \, x\) \in \oBB_\tau$ and 
$Q(x) \DEF x \in \oBB_0$ both satisfy conditions \eqref{E1.3} and
\eqref{E1.4} of Theorem~\ref{T1.1}, and in both cases the corresponding
$A_s(Q)$ is easy to compute. Therefore, inequalities \eqref{E1.6} and
\eqref{E1.7} follow immediately from \eqref{E1.5}. 

Next, we describe all extremal functions in Corollaries~\ref{C1.2} and
\ref{C1.3}.

Let $Q(x) \equiv \sin \(\tau \, x\)$ with $\tau > 0$. If $\sigma/\tau \in \NN$,
then $f(x) \DEF S \sin \((\sigma + \tau) x\) / \sin \(\tau \, x\)$ is an extremal
function for \eqref{E1.6} at $x = 0$. Let us assume that $f\in \oBB_\sigma$
with $f \not\equiv 0$ is an extremal function for \eqref{E1.6} at some $x =
x_0 \in \RR$ such that $\cos \(\tau \, x_0\) \ne 0$. Then, by
Theorem~\ref{T1.1}(b) and Remark~\ref{R1.3},
\begin{equation*}
	\sin \(\tau x\) f(x) \equiv
	S \sin \((\sigma + \tau) \, x\) + C \cos \((\sigma + \tau) \, x\) ,
	\qquad x \in \RR \, ,
\end{equation*}
where $|S| + |C| > 0$ are real. Clearly, $C$ must be $0$ since, otherwise,
$f$ is not even continuous at $0$. Furthermore, 
$\sin \(\(\sigma + \tau\) x\)/\sin \(\tau \, x\)$ belongs to $\oBB_\sigma$ if
and only if the zeros of the denominator are also zeros of the numerator.
Then the numerator must vanish at $\pi/\tau$ which is the first positive zero
of the denominator so that $\sigma$ must be an integer multiple of $\tau$.
This proves Corollary~\ref{C1.2}.

Next let $Q(x) \equiv x$. Then
$f(x) \DEF S \sin \(\sigma \, x\) / x$ is an extremal function for
\eqref{E1.7}
at $x = 0$. Moreover, if $f\in \oBB_\sigma$ with $f \not\equiv 0$ is an
extremal function for \eqref{E1.7} at some $x = x_0 \in \RR$, then, since $Q'
\equiv 1 \ne 0$, by Theorem \ref{T1.1}(b) and Remark~\ref{R1.3},
\begin{equation*}
	x f(x) \equiv
	S \sin \(\sigma \, x\) + C \cos \(\sigma \, x\) ,
	\qquad x \in \RR \, ,
\end{equation*}
where $|S| + |C| > 0$ are real. As before, $C$ must be $0$. Thus, Corollary
\ref{C1.3} has been proved as well.
\end{proof}


\begin{thebibliography}{99}

\bibitem{DS1937}

R. J. Duffin and A. C. Schaeffer, Some inequalities concerning functions
of exponential type, Bull. Amer. Math. Soc., {47} (1937), pp.~554--556.

\bibitem{katsnelson1973}

V. E. Katsnelson (aka Kacnel'son), Equivalent norms in spaces of entire
functions, Mat. Sb. (N.S.), {92(134)} (1973), pp.~34--54, in
Russian; English translation in Math. USSR-Sb., {21} (1973),
pp.~33--55.\footnote{Our page references use the English version.}

\bibitem{nevai2014}

Paul Nevai and The Anonymous Referee, The Bernstein Inequality and the Schur
Inequality are equivalent, J. Approx. Theory, 182 (2014), pp.~103--109.
\texttt{http://dx.doi.org/10.1016/j.jat.2014.02.006}

\bibitem{nevai201x}

Paul Nevai, \emph{The true story of $n$ vs. $2n$ in the Bernstein
Inequality\/}, book in progress.

\bibitem{rahman2002}

Q.~I.~Rahman and G.~Schmei{\ss}er, \emph{Analytic Theory of Polynomials\/},
London Math. Soc. Monographs, New Series, vol.~26, Calderon Press, Oxford,
2002.

\bibitem{mriesz1914_2}

Marcel Riesz, Eine trigonometrische Interpolationsformel und einige
Ungleichungen f\"ur Polynome, Jahresber. Dtsch. Math.-Ver., 23 (1914),
pp.~354--368.

\bibitem{schur1919}

I.~Schur, \"Uber das Maximum des absoluten Betrages eines Polynoms in einem
gegebenen Intervall,  Math. Z., 4:3--4 (1919),
pp.~271--287.

\end{thebibliography}
\end{document}